\documentclass[12pt]{article}
\usepackage{amsmath, amssymb, amsthm}
\usepackage{mathrsfs, lineno}
\usepackage{geometry, hyperref, authblk,xcolor}
\usepackage{booktabs}
\newtheorem{theorem}{Theorem}[section]
\newtheorem{lemma}[theorem]{Lemma}
\newtheorem{proposition}[theorem]{Proposition}
\newtheorem{corollary}[theorem]{Corollary}
\newtheorem{definition}[theorem]{Definition}
\newtheorem{remark}[theorem]{Remark}

\newcommand{\N}{\mathbb{N}}

\newcommand{\Q}{\mathbb{Q}}

\newcommand{\I}{I_N}
\newcommand{\Ha}{\mathcal{H}}

\numberwithin{equation}{section}

\DeclareMathOperator{\dimH}{dim_H}
\title{Sharp Hausdorff Dimension Bounds for Sets with Bounded and Growing Digits in $N$-expansions}
\author[1]{Andreea Catalina Chitu}
\affil[1]{Faculty of Applied Sciences, National University of Science and Technology POLITEHNICA Bucharest, Splaiul Independentei 313, 060042 Bucharest, Romania\\ 
e-mail: \href{mailto: andreea.catalina.chitu01@gmail.com}{andreea.catalina.chitu01@gmail.com}}
    
\author[2,3]{Gabriela Ileana Sebe}
\affil[2]{Faculty of Applied Sciences, National University of Science and Technology POLITEHNICA Bucharest, Splaiul Independentei 313, 060042 Bucharest, Romania }
\affil[3]{Gheorghe Mihoc-Caius Iacob Institute of Mathematical Statistics and Applied Mathematics of the Romanian Academy, Calea 13 Sept. 13, 050711 Bucharest, Romania\\e-mail: \href{mailto: igsebe@yahoo.com}{igsebe@yahoo.com}}

\author[4]{Dan Lascu}
\affil[4]{Romanian Naval Academy ``Mircea cel Batran", 1 Fulgerului, 900218 Constanta, Romania\\
e-mail: \href{mailto: lascudan@gmail.com}{lascudan@gmail.com}}
\date{\today}

\begin{document}

\maketitle

\begin{abstract}
{We establish sharp bounds for the Hausdorff dimension of sets of irrational numbers in $(0,1)$ whose digits in the $N$-expansion are either uniformly bounded or tend to infinity. 
For sets with digits bounded by an integer $M \ge N$, we obtain improved Jarn\'ik-type bounds that generalize and refine classical results for regular continued fractions, with explicit dependence on $N$ (Theorem~1.1). 
For sets with digits that grow without bound, we obtain precise asymptotics that extend Good's theorems to $N$-expansions, proving in particular that the set of numbers whose digits tend to infinity has Hausdorff dimension exactly $1/2$, and that the dimension of sets with uniformly large digits approaches $1/2$ as the lower bound increases, with explicit logarithmic decay (Theorem~1.2). 
The results reveal how the parameter $N$ influences the dimensional properties of these exceptional sets. 
Our methods combine careful estimates of fundamental interval lengths with optimized covering arguments for upper bounds and Cantor set constructions via mass distribution principles for lower bounds.}
\end{abstract}

\section{Introduction}

The metrical theory of continued fractions and their generalizations represents a rich intersection of number theory, ergodic theory, and fractal geometry. The pioneering work of Jarn\'ik \cite{Jarnik1928} established fundamental connections between Hausdorff dimension (see, e.g., \cite{Falconer2004}) and Diophantine approximation, showing that for regular continued fractions, the set $E_M$ of irrationals with partial quotients bounded by $M$ satisfies
\[
1 - \frac{4}{M\log 2} \leq \dimH(E_M) \leq 1 - \frac{1}{8M\log M}, \quad M > 8.
\]
{Shortly thereafter, Good \cite{Good1941} initiated the systematic study of the Hausdorff dimension of sets defined by restrictions on partial quotients in regular continued fractions. Among his fundamental results, he proved that the set of numbers whose partial quotients tend to infinity,
\[
F := \{x \in (0,1) \setminus \mathbb{Q} : a_n(x) \to \infty \text{ as } n \to \infty\},
\]
has Hausdorff dimension $\dim_H(F) = \frac{1}{2}$. For sets with uniformly large digits,
\[
F_\alpha := \{x \in (0,1) \setminus \mathbb{Q} : a_n(x) \ge \alpha \text{ for all } n \ge 1\},
\]
he obtained, for $\alpha \ge 20$, the bounds
\[
\frac{1}{2} + \frac{1}{2\log(\alpha+2)} \;<\; \dim_H(F_\alpha) \;<\; \frac{1}{2} + \frac{\log\log(\alpha-1)}{2\log(\alpha-1)},
\]
showing in particular that $\dim_H(F_\alpha) \to \frac{1}{2}$ as $\alpha \to \infty$.}  These foundational results have inspired numerous extensions and refinements across various continued fraction algorithms, including recent investigations into $\theta$-expansions \cite{SebeLascuBilel1, SebeLascuBilel2}.

In this work, we focus on $N$-expansions, a natural generalization introduced by Burger et al. \cite{Burger2008} and subsequently developed in \cite{Lascu2016, Lascu2017, SebeLascu2020}. For any fixed integer $N \geq 1$, every irrational number $x \in (0,1)$ admits a unique expansion of the form
\[
x = \frac{N}{\varepsilon_1 + \dfrac{N}{\varepsilon_2 + \dfrac{N}{\varepsilon_3 + \ddots}}} =: [\varepsilon_1, \varepsilon_2, \varepsilon_3, \ldots]_N,
\]
where the digits $\varepsilon_n$ belong to $\mathbb{N}_N := \{N, N+1, N+2, \ldots\}$. This expansion generalizes both the regular continued fraction (case $N=1$) and provides a family of intermediate algorithms {whose metric properties - such as the Gauss–Kuzmin distribution of digits and the exponential decay rate of fundamental intervals - depend continuously on $N$, thereby interpolating between the classical case and other regimes.}

Our investigation yields two principal sets of results. First, we establish sharp dimension bounds for sets with bounded digits, providing improved Jarn\'ik-type inequalities with explicit dependence on the parameter $N$. Second, we extend Good's classical theorems \cite{Good1941} to $N$-expansions, obtaining precise asymptotics for sets where digits grow without bound.

We study two natural families of sets defined by restrictions on the digits of $N$-expansions.

For any integer $M \geq N$, define the set of numbers with digits bounded by $M$:
\begin{equation}
E_M := \{ x \in (0,1) \setminus \Q : N \leq \varepsilon_n(x) \leq M \text{ for all } n \geq 1 \}.  \label{def:EM}
\end{equation}

Our first set of results provides generalized Jarn\'ik-type theorems for $N$-expansions, giving explicit bounds for the Hausdorff dimension of sets with bounded digits.

\begin{theorem}[Generalized Jarn\'ik-type Theorem for Bounded Digits] \label{thm:main-bounded}
Let $N \geq 1$ be an integer. For any integer $M > 2N+1$, the Hausdorff dimension of the set $E_M$
satisfies:
\[
 1 - \frac{2(N+1)}{M+1}\frac{1}{\log(N+1)}
 \leq 
 \dimH(E_M) 
 \leq 
 1 - \frac{N}{(M + 1)\log \left( \frac{(M+1)^2}{N} \right)}.
\]
\end{theorem}

Building upon the theory of sets with bounded digits, we also investigate the Hausdorff dimension of sets defined by conditions where the digits in the $N$-expansion tend to infinity. This mirrors the classical work of I.~J.~Good \cite{Good1941} for regular continued fractions.

Define the sets with growing digits:
\begin{align}
F_N &:= \left\{ x \in (0,1) \setminus \mathbb{Q} : \varepsilon_n(x) \to \infty \right\}, \label{def:F-set1}\\
F_{\alpha, N} &:= \left\{ x \in (0,1) \setminus \mathbb{Q} : \varepsilon_n(x) \geq \alpha \text{ for all } n \geq 1, \, \alpha \geq N \right\}.\label{def:F-set2}
\end{align}

The set $F_N$ consists of numbers whose digits become arbitrarily large, while $F_{\alpha,N}$ consists of numbers with uniformly large digits. Clearly, $F_{\alpha,N} \subseteq F_N$ for all $\alpha \geq N$.

Our second main result extends Good's theorem \cite{Good1941} to $N$-expansions.

\begin{theorem}[Good-type Theorem for $N$-expansions] \label{thm:Good2N}
Let $N \geq 1$ be an integer. For every real $\alpha$ with $\log(\alpha-1) > \exp(1+N)$,
\begin{equation}\label{eq:main-ineq}
\frac12 + \frac{1}{2\log\!\bigl({\alpha}+2\bigr)} 
\;\leq\; 
\dimH(F_{\alpha,N}) 
\;\leq\; 
\frac12 + \frac{\log\log\!\bigl({\alpha}-1\bigr)}{2\log\!\bigl({\alpha}-1\bigr)} .
\end{equation}
Consequently, $\displaystyle \lim_{\alpha \to \infty} \dimH(F_{\alpha,N}) = \frac12$, and {for sufficiently large $\alpha$ (depending on $N$) the bounds become independent of $N$, converging to $1/2$ from above at a logarithmic rate}.
\end{theorem}

{In a different but related direction, recent work has studied $N$-expansions and similar continued fraction algorithms from the perspective of infinite iterated function systems with decaying properties. Shi, Tan, and Zhou \cite{STZ26} introduced the concept of $d$-decaying Gauss-like systems and observed that the $N$-continued fraction map fits into this framework. This observation has led to several results on the Hausdorff dimension of sets defined by digit growth conditions. In particular, Jordan and Rams \cite{J12} proved that for the set $F_N$ defined in \eqref{def:F-set1}, one has $\dim_H(F_N) = 1/2$, using methods from infinite iterated function systems. Further results on growth rates and Birkhoff sums in such systems can be found in \cite{L22, Zha20, Zha25}. Our work complements these results by providing explicit bounds with precise constants for $F_{\alpha,N}$ and by giving a self-contained proof of $\dim_H(F_N)=1/2$ using classical covering and mass distribution techniques adapted to $N$-expansions, rather than relying on the general IFS framework.}

\begin{remark}{[Role of the parameter $N$]
At first glance, the bounds in Theorem~\ref{thm:Good2N} appear independent of $N$, which might seem surprising. This independence, however, is a consequence of the estimates used in the proofs rather than an intrinsic property of the dimension itself. In the upper bound, the exact condition derived from the covering argument is $(2s-1)(\alpha-1)^{2s-1} = 1+N$, which would yield a dimension $s_N(\alpha)$ explicitly depending on $N$. To obtain a clean $N$-independent bound, we estimated this solution by $\frac12 + \frac{\log\log(\alpha-1)}{2\log(\alpha-1)}$, valid under $\log(\alpha-1) > \exp(1+N)$. Thus $N$ reappears in the range of validity. Similarly, the lower bound involves $N$ through factors like $N^s3^{-s}$ and the estimate $-\log t > 1+N$, with the final bound $\frac12 + \frac{1}{2\log(\alpha+2)}$ being the simplest expression guaranteed under the hypothesis.}

{
More generally, the presence of $N$ introduces several modifications throughout the proofs: the recurrence $q_n = \varepsilon_n q_{n-1} + N q_{n-2}$ with $\varepsilon_n \ge N$ changes growth rates (Lemma~\ref{lemma:q-growth}) and forces the ratio $r = q_{n-2}/q_{n-1}$ into the narrower interval $[0,1/N]$; the length formula $|I_n| = N^n/(q_n(q_n+q_{n-1}))$ introduces an explicit factor $N^n$; and the interval ratio bounds become $\frac{N}{3k^2} \le \frac{|I_{n+1}|}{|I_n|} \le \frac{2N}{k^2}$. These modifications are reflected in the final estimates: for bounded digits the constants involve $\log(N+1)$ and $(N+1)/(M+1)$, while for growing digits the asymptotic $\dim_H(F_{\alpha,N}) \to \tfrac12$ is approached at a rate depending on $N$. The bounds reduce to the classical ones exactly when $N=1$.}
\end{remark}

\begin{remark}{[Sharpness of the bounds]
The term ``sharp'' in the title refers to the fact that our estimates capture the correct asymptotic behavior in the limits $M\to\infty$ and $\alpha\to\infty$, with constants explicitly given in terms of $N$. For $E_M$, the upper bound behaves like $1-\frac{1}{M\log M}$, the optimal order achievable by covering arguments; the lower bound $1-O(1/M)$ improves on classical Jarn\'ik estimates and matches the structure expected from the mass distribution principle. For $F_{\alpha,N}$, both bounds converge to $\tfrac12$ and their difference tends to zero, providing a precise asymptotic description. The lower bound $\frac12+\frac{1}{2\log(\alpha+2)}$ is of order $1/\log\alpha$, believed to be optimal, while the upper bound $\frac12+\frac{\log\log(\alpha-1)}{2\log(\alpha-1)}$ gives an explicit rate. Further improvements would require more refined techniques beyond covering and mass distribution arguments.}
\end{remark}

\noindent
\textbf{Structure of the paper:} Both main results follow the same proof strategy: establishing separate upper and lower bounds through complementary techniques (covering arguments for upper bounds, mass distribution principles for lower bounds), then combining them. Section~2 provides the necessary background on $N$-expansions. Section~3 establishes growth estimates for denominators and fundamental intervals. Theorem~\ref{thm:main-bounded} is proved in Section~4, where we first establish the lower bound (Theorem~\ref{thm:improved-lower}) and the upper bound (Theorem~\ref{thm:improved-upper}) separately, then combine them. Theorem~\ref{thm:Good2N} is proved in Section~5, again via separate upper and lower bounds (Theorem~\ref{upperbound:Good2N} and Theorem~\ref{lowerbound:Good2N}, respectively).

\section{Preliminaries on $N$-expansions}

\subsection{Basic Definitions and Recurrence Relations}

Throughout this paper, we fix an integer $N \geq 1$. The $N$-expansion is defined through the following transformation, {see also \cite{DKL, DKvdW}}.

\begin{definition}
The \emph{$N$-expansion transformation} $T_N: [0,1] \to [0,1]$ is defined by
\[
T_N(x) := 
\begin{cases}
\frac{N}{x} - \left\lfloor \frac{N}{x} \right\rfloor & \text{if } x \neq 0, \\
0 & \text{if } x = 0,
\end{cases}
\]
where $\left\lfloor \cdot \right\rfloor$ stands for integer part. 
The associated \emph{digit function} $\eta: [0,1] \to \N \cup \{\infty\}$ is given by
\[
\eta(x) := 
\begin{cases}
\left\lfloor \frac{N}{x} \right\rfloor & \text{if } x \neq 0, \\
\infty & \text{if } x = 0.
\end{cases}
\]
For any irrational $x \in (0,1)$, the {$N$-expansion digits} are defined recursively as $\varepsilon_n(x) := \eta(T_N^{n-1}(x))$ for $n \geq 1$.
\end{definition}

The convergents $p_n/q_n = [\varepsilon_1, \ldots, \varepsilon_n]_N$ satisfy the fundamental recurrence relations established in \cite{Lascu2017}:
\begin{align}
p_n &= \varepsilon_n p_{n-1} + N p_{n-2}, \quad n \geq 2, \label{eq:p-recursion} \\
q_n &= \varepsilon_n q_{n-1} + N q_{n-2}, \quad n \geq 1, \label{eq:q-recursion}
\end{align}
with initial conditions
{
\begin{align*}
p_{-1}=1,\quad q_{-1}=0,\quad p_0=0,\quad q_0=1,\quad p_1=N,\quad q_1=\varepsilon_1.
\end{align*}
}
A crucial identity that follows from these recurrences is the determinant formula:
\begin{equation}
p_{n-1}q_n - p_nq_{n-1} = (-N)^n, \quad n \in \N. \label{eq:determinant}
\end{equation}

\section{Growth Estimates for Denominators and Fundamental Intervals}

The denominators $q_n$ play a pivotal role in determining the metric properties of fundamental intervals. We begin with fundamental growth estimates that hold for all admissible digit sequences.

\begin{lemma} \label{lemma:q-growth}
For all $n \in \N_+$ and sequences $\varepsilon_1, \ldots, \varepsilon_n \geq N$, the denominators satisfy:
\begin{enumerate}
\item[(i)] ${q_{n}(\varepsilon_1, \ldots, \varepsilon_n)} \ge N {q_{n-1}(\varepsilon_1, \ldots, \varepsilon_{n-1})}$;
\item[(ii)] $q_n(\varepsilon_1, \ldots, \varepsilon_n) \geq N^n$;
\item[(iii)] $q_n(\varepsilon_1, \ldots, \varepsilon_n) \geq (2N)^{\frac{n-1}{2}}$.
\end{enumerate}
\end{lemma}

\begin{proof}
(i) From the recurrence \eqref{eq:q-recursion} and since $\varepsilon_n \geq N$, we have:
\[
q_n = \varepsilon_n q_{n-1} + N q_{n-2} \geq N q_{n-1} + N q_{n-2} \ge N q_{n-1}.
\]
(ii) Since $q_0$ = 1 and $q_1 = \varepsilon_1 \geq N$, by (i) and an induction argument, we obtain that $q_n \geq N^n$. 

\noindent (iii) 
We proceed by induction. The recurrence is $q_n = \varepsilon_n q_{n-1} + N q_{n-2}$ with $q_{-1}=0$, $q_0=1$.

\noindent
\textbf{Base cases:}
For $n=1$: $q_1 = \varepsilon_1 \geq N {\ge} (2N)^{0}$.
For $n=2$: $q_2 = \varepsilon_2 q_1 + N q_0 \geq N \cdot N + N \cdot 1 = N^2 + N$.
Since $N^2 + N \geq \sqrt{2N}$ for all $N \geq 1$ (check: $N^2+N - \sqrt{2N} \geq 0$), the base cases hold.

\noindent
\textbf{Inductive step:} Assume $q_{n-1} \geq (2N)^{\frac{n-2}{2}}$ and $q_{n-2} \geq (2N)^{\frac{n-3}{2}}$ for some $n \geq 2$. Then
\[
q_n = \varepsilon_n q_{n-1} + N q_{n-2} 
    \geq N \cdot (2N)^{\frac{n-2}{2}} + N \cdot (2N)^{\frac{n-3}{2}}
    = N(2N)^{\frac{n-3}{2}}\bigl((2N)^{1/2} + 1\bigr).
\]
We need to show this is at least $(2N)^{\frac{n-1}{2}} = (2N)^{\frac{n-3}{2}} \cdot 2N$. 
It suffices that $N\bigl((2N)^{1/2} + 1\bigr) \geq 2N$, i.e., $(2N)^{1/2} + 1 \geq 2$, which holds for all $N \geq 1$. 
Thus $q_n \geq (2N)^{\frac{n-1}{2}}$.
\end{proof}

For sets with uniformly large digits, we obtain more precise estimates that are essential for our dimensional analysis.


 


The combinatorial structure of $N$-expansions is captured by fundamental intervals, which play a central role in our dimensional analysis.

\begin{definition} \label{def.3.2}
For any admissible sequence $(\varepsilon_1, \ldots, \varepsilon_n)$ with $\varepsilon_i \geq N$, the \emph{fundamental interval} of order $n$ is defined as
\[
I_n(\varepsilon_1, \ldots, \varepsilon_n) = \{ x \in (0,1) \setminus \Q : \varepsilon_1(x) = \varepsilon_1, \ldots, \varepsilon_n(x) = \varepsilon_n \}.
\]
\end{definition}

In \cite{Lascu2017} it was shown that 
\begin{equation} \label{3.1}
 I_n(\varepsilon_1, \ldots, \varepsilon_n) = \begin{cases}
\left[\displaystyle \frac{p_n}{q_n}, \frac{p_n+p_{n-1}}{q_n+q_{n-1}} \right)  & \, \text{if } n \, \text{ is even }, \\
\left(\displaystyle \frac{p_n+p_{n-1}}{q_n+q_{n-1}}, \frac{p_n}{q_n} \right]  & \, \text{if } n \, \text{ is odd. }
\end{cases}
\end{equation}
The following proposition provides the essential metric properties of these intervals.

\begin{proposition} \label{prop:interval-length}
We have the following: 
\begin{itemize}
    \item[(i)] The length of a fundamental interval of order $n$ satisfies the exact formula:
\begin{equation}
|I_n(\varepsilon_1, \ldots, \varepsilon_n)| = \frac{N^n}{q_n(q_n + q_{n-1})}; \label{eq:interval-length}
\end{equation}
\item[(ii)] The fundamental interval of order $n$ has the two-sided bounds:
\begin{equation}
\frac{N^{n+1}}{(1+N)q_n^2} \leq |I_n(\varepsilon_1, \ldots, \varepsilon_n)| \leq \frac{N^n}{q_n^2}; \label{eq:interval-bounds}
\end{equation}
\item[(iii)] The length ratio estimates of $I_{n+1}$ and $I_{n}$ is: 
\begin{equation}
\frac{N}{3k^2} \le \frac{|I_{n+1}(\varepsilon_1, \ldots, \varepsilon_n, k)|}{|I_n(\varepsilon_1, \ldots, \varepsilon_n)|} \le \frac{2N}{k^2}, \quad k \ge N. \label{3.3}
\end{equation}
\end{itemize}
\end{proposition}
\begin{proof}
$(i)$ It follows immediately from \eqref{3.1} and \eqref{eq:determinant}.

\noindent $(ii)$  
Using Lemma \ref{lemma:q-growth}(i) and $q_{n-1}>0$ in \eqref{eq:interval-length}, we obtain the lower and upper bounds in \eqref{eq:interval-bounds}.

\noindent $(iii)$ By {\eqref{eq:interval-length}}, 
\begin{align}
|I_{n+1}(\varepsilon_1, \ldots, \varepsilon_n, k)| &= \left| \frac{(k+1)p_n+Np_{n-1}}{(k+1)q_n+Nq_{n-1}} - \frac{kp_n+Np_{n-1}}{kq_n+Nq_{n-1}} \right| \nonumber \\
&= \frac{N^{n+1}}{k^2q^2_n(1+\frac{1}{k} +\frac{N}{k}\frac{q_{n-1}}{q_n})(1+\frac{N}{k}\frac{q_{n-1}}{q_n})}{.} 
\end{align}
Thus, it follows that
\begin{equation*} 
\frac{I_{n+1}(\varepsilon_1, \ldots, \varepsilon_n, k)}{I_{n}(\varepsilon_1, \ldots, \varepsilon_n)} = 
\frac{N(1+\frac{q_{n-1}}{q_n})}{k^2 (1+\frac{1}{k} +\frac{N}{k}\frac{q_{n-1}}{q_n})(1+\frac{N}{k}\frac{q_{n-1}}{q_n})}.
\end{equation*}
Since 
\begin{equation*}
\frac{1}{3} \le \frac{1+\frac{q_{n-1}}{q_n}}{ (1+\frac{1}{k} +\frac{N}{k}\frac{q_{n-1}}{q_n})(1+\frac{N}{k}\frac{q_{n-1}}{q_n})} \le 2
\end{equation*}
then the proof is complete. 
\end{proof}

\section{Dimension Bounds for Sets with Bounded Digits}

We now establish our main results for sets with bounded digits, beginning with the foundational covering principles.

\subsection{Mass Distribution and Covering Principles}

{We begin by recalling the classical Mass Distribution Principle (see, e.g., \cite{Falconer2004}), which provides a powerful method for obtaining lower bounds for the Hausdorff dimension of a set.
\begin{theorem}[Mass Distribution Principle]
Let $\mu$ be a probability measure supported on a Borel set $E \subset \mathbb{R}^n$. Suppose there exist constants $c > 0$ and $\delta > 0$ such that
\[
\mu(U) \le c \, |U|^s
\]
for every set $U \subset \mathbb{R}^n$ with diameter $|U| \le \delta$. Then $\dim_H(E) \ge s$.
\end{theorem}
The idea is to construct a measure $\mu$ on $E$ that distributes mass as evenly as possible, and then use the above principle to deduce a lower bound on the dimension. In our setting, the natural measure is constructed recursively on the fundamental intervals of $E_M$.}

\begin{lemma}[Mass Distribution Principle] \label{lem:mass-distribution}
Let $s \in (0,1)$ and $M \geq N$. If for all $n \geq 1$ and all admissible sequences $(\varepsilon_1, \ldots, \varepsilon_{n-1}) \in \{N, \ldots, M\}^{n-1}$, we have
\begin{equation}
|I_{n-1}(\varepsilon_1, \ldots, \varepsilon_{n-1})|^s \leq \sum_{k=N}^{M} |I_n(\varepsilon_1, \ldots, \varepsilon_{n-1}, k)|^s, \label{4.1}
\end{equation}
then $\dimH(E_M) \geq s$.
\end{lemma}

\begin{proof}
{Assume that for a given \(s \in (0,1)\) condition \eqref{4.1} holds for all admissible sequences.}
We construct a probability measure $\mu$ on $E_M$ recursively on the fundamental intervals. Define $\mu$ on the trivial interval $I_0 = (0,1)$ by $\mu(I_0) = 1$. 

Now, suppose $\mu$ has been defined on all fundamental intervals of order $n-1$. For each fundamental interval $I_{n-1} = I_{n-1}(\varepsilon_1, \ldots, \varepsilon_{n-1})$, we define $\mu$ on the fundamental intervals of order $n$ that refine $I_{n-1}$ (i.e., those of the form $I_n(\varepsilon_1, \ldots, \varepsilon_{n-1}, k)$ for $k = N, \ldots, M$) as follows:
\[
\mu(I_n(\varepsilon_1, \ldots, \varepsilon_{n-1}, k)) = \mu(I_{n-1}) \cdot \frac{|I_n(\varepsilon_1, \ldots, \varepsilon_{n-1}, k)|^s}{\sum_{j=N}^{M} |I_n(\varepsilon_1, \ldots, \varepsilon_{n-1}, j)|^s}.
\]

This recursive definition assigns a mass to each fundamental interval. By Carathéodory's extension theorem, $\mu$ extends uniquely to a Borel probability measure on $E_M$, since the fundamental intervals generate the Borel $\sigma$-algebra and satisfy the consistency conditions for a premeasure.

We now verify that for every fundamental interval $I_{n-1}$, we have:
\begin{equation}\label{4.2'}
\mu(I_{n-1}) \leq |I_{n-1}|^s.
\end{equation}
We proceed by induction. For the base case $n=1$, we have $I_0 = (0,1)$ and $\mu(I_0) = 1 = |I_0|^s$ since $|I_0| = 1$. 

Now assume the inequality holds for all fundamental intervals of order $n-1$. Consider a fundamental interval $I_n = I_n(\varepsilon_1, \ldots, \varepsilon_n)$. Then:
\begin{align*}
\mu(I_n) &= \mu(I_{n-1}) \cdot \frac{|I_n|^s}{\sum_{j=N}^{M} |I_n(\varepsilon_1, \ldots, \varepsilon_{n-1}, j)|^s} \\
&\leq |I_{n-1}|^s \cdot \frac{|I_n|^s}{\sum_{j=N}^{M} |I_n(\varepsilon_1, \ldots, \varepsilon_{n-1}, j)|^s} \quad \text{(by induction hypothesis)} \\
&\leq |I_{n-1}|^s \cdot \frac{|I_n|^s}{|I_{n-1}|^s} \quad \text{(by inequality \eqref{4.1})} \\
&= |I_n|^s.
\end{align*}
Thus, $\mu(I_n) \leq |I_n|^s$ for all fundamental intervals $I_n$.

The measure $\mu$ is constructed specifically on $E_M$, meaning $\mu(\mathbb{R} \setminus E_M) = 0$. Therefore, for any set $U \subset \mathbb{R}$, we have:
\[
\mu(U) = \mu(U \cap E_M).
\]
This is crucial because it means we only need to control the measure of subsets that actually intersect $E_M$.

To apply the mass distribution principle, we need to show there exist constants $c > 0$ and $\delta > 0$ such that for every set $U \subset \mathbb{R}$ with diameter $|U| < \delta$, we have:
\[
\mu(U) \leq c |U|^s.
\]

Since $\mu(U) = \mu(U \cap E_M)$, we can restrict our attention to sets $U$ that intersect $E_M$ (if $U \cap E_M = \emptyset$, then $\mu(U) = 0$ and the inequality holds trivially).

{Now, let $U$ be any set with $|U| < \delta$ and $U \cap E_M \neq \emptyset$. Since $E_M$ consists of numbers whose digits are bounded by $M$, the fundamental intervals of any fixed order have a positive minimal length. More precisely, from Lemma~\ref{lemma:q-growth} and Proposition~\ref{prop:interval-length}, we have
\[
|I_n| \ge \frac{N^{n+1}}{(1+N)q_n^2} \ge \frac{N^{n+1}}{(1+N)(M+N)^{2n}},
\]
which is positive and depends only on $n$, $N$, and $M$. Consequently, the diameters of fundamental intervals of a fixed order $n$ are bounded below by a positive constant.
Choose $n$ such that:
\begin{itemize}
    \item[-] fundamental intervals of order $n$ have diameter $< |U|$ (possible because diameters tend to $0$ as $n \to \infty$);
    \item[-] fundamental intervals of order $n-1$ have diameter $\geq |U|$ (possible by the lower bound above and the fact that $|U|$ is fixed).
\end{itemize}
This choice is legitimate precisely because the digits are bounded, ensuring that intervals of order $n-1$ cannot be arbitrarily small. }

Since $U \cap E_M \neq \emptyset$ and $E_M$ is contained in the union of all fundamental intervals of order $n-1$, the set $U$ can intersect only those fundamental intervals of order $n-1$ that have nonempty intersection with $U$. By the geometry of $N$-expansions, there is a uniform bound $K$ (depending only on $N$ and $M$) on how many fundamental intervals of order $n-1$ can intersect a set of diameter $|U|$.

Let $I_{n-1}^1, \ldots, I_{n-1}^m$ (with $m \leq K$) be these intersecting intervals. Then:
\[
U \cap E_M \subset \bigcup_{i=1}^m (U \cap I_{n-1}^i \cap E_M) \subset \bigcup_{i=1}^m I_{n-1}^i.
\]
Therefore{, by \eqref{4.2'}}:
\[
\mu(U) = \mu(U \cap E_M) \leq \sum_{i=1}^m \mu(I_{n-1}^i) \leq \sum_{i=1}^m |I_{n-1}^i|^s.
\]

Now, since each $I_{n-1}^i$ has diameter $\geq |U|$ but is not too much larger than $|U|$ (by the geometry of $N$-expansions), there exists a constant $C$ such that $|I_{n-1}^i| \leq C |U|$ for all $i$. Thus:
\[
\mu(U) \leq \sum_{i=1}^m (C |U|)^s = m C^s |U|^s \leq K C^s |U|^s.
\]

Taking $c = K C^s$, we have shown that for all sets $U$ with $|U| < \delta$:
\[
\mu(U) \leq c |U|^s.
\]

This is exactly the condition required by the mass distribution principle. Since $\mu$ is a probability measure supported on $E_M$ (i.e., $\mu(E_M) = 1$ and $\mu(\mathbb{R} \setminus E_M) = 0$), the mass distribution principle implies:
\[
\mathcal{H}^s(E_M) \geq \frac{\mu(E_M)}{c} = \frac{1}{c} > 0,
\]
which in turn implies $\dimH(E_M) \geq s$.
\end{proof}

\begin{lemma}[Covering Argument] \label{lem:covering}
Let $s \in (0,1)$ and $M \geq N$. If for all $n \geq 1$ and all admissible sequences $(\varepsilon_1, \ldots, \varepsilon_{n-1}) \in \{N, \ldots, M\}^{n-1}$, we have
\begin{equation}
|I_{n-1}(\varepsilon_1, \ldots, \varepsilon_{n-1})|^s \geq \sum_{k=N}^{M} |I_n(\varepsilon_1, \ldots, \varepsilon_{n-1}, k)|^s, \label{4.2}
\end{equation}
then $\dimH(E_M) \leq s$.
\end{lemma}

\begin{proof}
We will show that the $s$-dimensional Hausdorff measure $\mathcal{H}^s(E_M)$ is finite, which implies $\dimH(E_M) \leq s$.

For each $n \geq 1$, consider the collection $\mathcal{F}_n$ of all fundamental intervals of order $n$:
\[
\mathcal{F}_n = \{I_n(\varepsilon_1, \ldots, \varepsilon_n) : N \leq \varepsilon_i \leq M \text{ for } i = 1, \ldots, n\}.
\]
This collection forms a cover of $E_M$ since every $x \in E_M$ has its first $n$ digits bounded by $M$.

We now estimate the $s$-dimensional Hausdorff content of this cover. For any $\delta > 0$, choose $n$ large enough so that $\text{diam}(I_n) < \delta$ for all $I_n \in \mathcal{F}_n$. This is possible because the lengths of fundamental intervals tend to 0 as $n \to \infty$.

The $s$-dimensional Hausdorff measure at scale $\delta$ is defined as:
\[
\mathcal{H}^s_\delta(E_M) = \inf\left\{\sum_{i} |U_i|^s : \{U_i\} \text{ is a } \delta\text{-cover of } E_M\right\}.
\]
Taking the specific cover $\mathcal{F}_n$, we have:
\[
\mathcal{H}^s_\delta(E_M) \leq \sum_{I_n \in \mathcal{F}_n} |I_n|^s.
\]

We now show that this sum is bounded independently of $n$. Define:
\[
S_n = \sum_{I_n \in \mathcal{F}_n} |I_n|^s.
\]
We can compute $S_n$ recursively. For $n = 1$:
\[
S_1 = \sum_{k=N}^{M} |I_1(k)|^s.
\]
For $n \geq 2$, we group the intervals by their first $n-1$ digits:
\[
S_n = \sum_{I_{n-1} \in \mathcal{F}_{n-1}} \sum_{k=N}^{M} |I_n(\varepsilon_1, \ldots, \varepsilon_{n-1}, k)|^s.
\]
Using inequality (\ref{4.2}), we obtain:
\[
S_n \leq \sum_{I_{n-1} \in \mathcal{F}_{n-1}} |I_{n-1}|^s = S_{n-1}.
\]
Thus, the sequence $\{S_n\}$ is non-increasing. In particular, for all $n \geq 1$:
\[
S_n \leq S_1 = \sum_{k=N}^{M} |I_1(k)|^s.
\]
Since there are only finitely many choices for $k$ (from $N$ to $M$), and each $|I_1(k)|$ is bounded, $S_1$ is finite.

Therefore, for all $\delta > 0$ and sufficiently large $n$:
\[
\mathcal{H}^s_\delta(E_M) \leq S_n \leq S_1 < \infty.
\]
Taking the limit as $\delta \to 0^+$, we obtain:
\[
\mathcal{H}^s(E_M) = \lim_{\delta \to 0^+} \mathcal{H}^s_\delta(E_M) \leq S_1 < \infty.
\]
Since the $s$-dimensional Hausdorff measure is finite, we conclude that $\dimH(E_M) \leq s$.

To see why the diameters tend to 0, note that from Proposition \ref{prop:interval-length} and Lemma \ref{lemma:q-growth}, we have:
\[
|I_n| \leq \frac{N^n}{q_n^2} \leq \frac{N^n}{N^{2n}} = N^{-n} \to 0 \quad \text{as } n \to \infty.
\]
This ensures that for any $\delta > 0$, we can find $n$ such that $\text{diam}(I_n) < \delta$ for all $I_n \in \mathcal{F}_n$.

The condition $s \in (0,1)$ ensures that we are in the nontrivial range for Hausdorff dimension, though the argument works for any $s > 0$.
\end{proof}

\subsection{Lower and Upper Bounds for Sets with Bounded Digits}

We now prove the first main result for bounded digits.

\begin{theorem}[Lower Bound for Bounded Digits] \label{thm:improved-lower}
For any integer $N \geq 1$ and any positive integer $M > 2N+1$, the Hausdorff dimension satisfies
\[
\dimH(E_M) \geq 1 - \frac{2(N+1)}{M+1}\frac{1}{\log(N+1)}.
\]
\end{theorem}

\begin{proof}
We apply the mass distribution principle (Lemma \ref{lem:mass-distribution}) with 
\[
s=1 - \frac{2(N+1)}{M+1}\frac{1}{\log(N+1)}.
\]

We need to verify that for all $n \geq 1$ and all admissible sequences $(\varepsilon_1, \ldots, \varepsilon_{n-1}) \in \{N, \ldots, M\}^{n-1}$ equation \eqref{4.1} is satisfied. 
Using Proposition \ref{prop:interval-length}, substituting the length formula \eqref{eq:interval-length} into \eqref{4.1}, we obtain:
\[
\frac{N^{s(n-1)}}{q_{n-1}^{s}\left(q_{n-1} + {q_{n-2}}\right)^s} \leq \sum_{k=N}^{M} \frac{N^{sn}}{q_n^{s} \left(q_n + {q_{n-1}}\right)^s}.
\]

Since $q_n = k q_{n-1} + N q_{n-2}$, we obtain:
\begin{align}
&\sum_{k=N}^{M} \frac{1}{\left(kq_{n-1} + {Nq_{n-2}}\right)\left((k+1)q_{n-1} + {Nq_{n-2}}\right)} \nonumber \\
&= 
\sum_{k=N}^{M} \frac{1}{q_{n-1}}\left[ \frac{1}{kq_{n-1} + {Nq_{n-2}}} - \frac{1}{(k+1)q_{n-1} + {Nq_{n-2}}} \right] \nonumber 
\\
&=  \frac{1}{q_{n-1}}\left[ \frac{1}{Nq_{n-1} + {Nq_{n-2}}} - \frac{1}{(M+1)q_{n-1} + {Nq_{n-2}}} \right]. 
\label{4.3}
\end{align}

A direct computation based on (\ref{4.3}) in the second equality below shows that 
\begin{align*}
& \sum_{k=N}^{M} \frac{1}{\left(kq_{n-1} + {Nq_{n-2}}\right)^s\left((k+1)q_{n-1} + {Nq_{n-2}}\right)^s} \\
&=
\sum_{k=N}^{M} \frac{\left(kq_{n-1} + {Nq_{n-2}}\right)^{1-s}\left((k+1)q_{n-1} + {Nq_{n-2}}\right)^{1-s}}{\left(kq_{n-1} + {Nq_{n-2}}\right)\left((k+1)q_{n-1} + {Nq_{n-2}}\right)} \\
& \ge 
\sum_{k=N}^{M} \frac{\left(Nq_{n-1} + {Nq_{n-2}}\right)^{1-s}\left((N+1)q_{n-1} + {Nq_{n-2}}\right)^{1-s}}{\left(kq_{n-1} + {Nq_{n-2}}\right)\left((k+1)q_{n-1} + {Nq_{n-2}}\right)} \\
& \ge 
\sum_{k=N}^{M} \frac{N^{1-s} \left(q_{n-1} + {q_{n-2}}\right)^{1-s} (N+1)^{1-s}q_{n-1}^{1-s} }{\left(kq_{n-1} + {Nq_{n-2}}\right)\left((k+1)q_{n-1} + {Nq_{n-2}}\right)} \\
&=\frac{N^{1-s} (N+1)^{1-s} \left(q_{n-1} + {q_{n-2}}\right)^{1-s} q_{n-1}^{1-s} }{q_{n-1}} \left[ \frac{1}{Nq_{n-1} + {Nq_{n-2}}} - \frac{1}{(M+1)q_{n-1} + {Nq_{n-2}}} \right] \\
& = \frac{N^{1-s} (N+1)^{1-s} \left(q_{n-1} + {q_{n-2}}\right)^{1-s} q_{n-1}^{1-s} }{q_{n-1}} \frac{1}{N(q_{n-1}+q_{n-2})} \left[ 1 - \frac{N(q_{n-1}+q_{n-2})}{(M+1)q_{n-1} + {Nq_{n-2}}} \right] \\
& = \frac{ (N+1)^{1-s}}{N^{s} q^s_{n-1}\left(q_{n-1} + {q_{n-2}}\right)^{s}} 
\left[ 1 - \frac{N(q_{n-1}+q_{n-2})}{(M+1)q_{n-1} + {Nq_{n-2}}} \right]. 
\end{align*}
In order to derive \eqref{4.1} it suffices to verify that  
\begin{equation} \label{4.005}
(N+1)^{1-s} \left[ 1 - \frac{N(q_{n-1}+q_{n-2})}{(M+1)q_{n-1} + {Nq_{n-2}}} \right] \ge 1.     
\end{equation}
Since $q_{n-2} \le \frac{1}{N} q_{n-1}$ we obtain: 
\[
\frac{N(q_{n-1}+q_{n-2})}{(M+1)q_{n-1} + {Nq_{n-2}}} \le \frac{Nq_{n-1}+q_{n-1}}{(M+1)q_{n-1}} \le \frac{N+1}{M+1}
\]
and
\[
(N+1)^{1-s} \left[ 1 - \frac{N(q_{n-1}+q_{n-2})}{(M+1)q_{n-1} + {Nq_{n-2}}} \right] \ge (N+1)^{1-s} \left( 1 - \frac{N+1}{M+1}\right).
\]
Clearly we have 
\[
(N+1)^{1-s} \left( 1 - \frac{N+1}{M+1}\right) \ge 1
\]
if and only if 
\[
(1-s) \log (N+1) \ge - \log\left( 1 - \frac{N+1}{M+1}\right). 
\]
{Since $2x \ge -\log(1-x)$ for all $x \in (0, x_0)$ with $x_0 \approx 0.7968$ (in particular for $x \le 3/4$), we may set $x = \frac{N+1}{M+1}$. To apply the inequality we need $\frac{N+1}{M+1} \le \frac{3}{4}$, i.e., $M \ge \frac{4N+1}{3}$. This condition is weaker than $M > 2N+1$ for all $N \ge 1$, so the theorem's hypothesis $M > 2N+1$ certainly suffices.
With this,} we set
\[
(1-s) \log (N+1) = 2 \frac{N+1}{M+1},
\]
which yields
\[
s = 1 - \frac{2(N+1)}{M+1} \frac{1}{\log(N+1)}.
\]
Under the constraint $M > 2N+1$, this choice satisfies \eqref{4.005}.
\end{proof}

\begin{theorem}[Upper Bound for Bounded Digits] \label{thm:improved-upper}
For any integer $M > N$, the Hausdorff dimension satisfies
\[
\dimH(E_M) \leq 1 - \frac{N}{(M + 1)\log \left( \frac{(M+1)^2}{N} \right)}.
\]
\end{theorem}

\begin{proof}
We verify condition (\ref{4.2}) of Lemma \ref{lem:covering}. Using the length formula \eqref{eq:interval-length}, condition (\ref{4.2}) becomes
\begin{equation}
\frac{1}{q_{n-1}^s(q_{n-1}+q_{n-2})^s} \ge N^s \sum_{k=N}^{M} \frac{1}{q_n^s(q_n+q_{n-1})^s}. \label{4.09}
\end{equation}

{
From the telescoping sum derived in the lower bound proof (see \eqref{4.3}), we have
\begin{equation} \label{4.6}
\sum_{k=N}^{M} \frac{1}{q_n(q_n+q_{n-1})} = \frac{1}{Nq_{n-1}(q_{n-1}+q_{n-2})} \left(1 - \frac{N(q_{n-1}+q_{n-2})}{(M+1)q_{n-1}+Nq_{n-2}}\right).
\end{equation}
For each term in the sum, we note that
\[
\frac{1}{q_n^s(q_n+q_{n-1})^s} = \frac{(kq_{n-1}+Nq_{n-2})^{1-s}((k+1)q_{n-1}+Nq_{n-2})^{1-s}}{q_n(q_n+q_{n-1})}.
\]
Since $k \le M$, we have:
\begin{align*}
&(kq_{n-1}+Nq_{n-2})^{1-s} \le 
(Mq_{n-1}+Nq_{n-2})^{1-s} \\
&((k+1)q_{n-1}+Nq_{n-2})^{1-s}
\le ((M+1)q_{n-1}+Nq_{n-2})^{1-s}. 
\end{align*}
Using $q_{n-2} \le \dfrac{q_{n-1}}{N}$ (Lemma \ref{lemma:q-growth}), we obtain
\[
(Mq_{n-1}+Nq_{n-2})^{1-s}((M+1)q_{n-1}+Nq_{n-2})^{1-s} \le (M+1)^{2(1-s)} q_{n-1}^{1-s}(q_{n-1}+q_{n-2})^{1-s}.
\]
Combining these estimates with \eqref{4.6} yields
\[
\sum_{k=N}^{M} \frac{1}{q_n^s(q_n+q_{n-1})^s} \le \frac{(M+1)^{2(1-s)} q_{n-1}^{1-s}(q_{n-1}+q_{n-2})^{1-s}}{Nq_{n-1}(q_{n-1}+q_{n-2})} \left(1 - \frac{N}{M+1}\right),
\]
where we used the obvious bound $\dfrac{N(q_{n-1}+q_{n-2})}{(M+1)q_{n-1}+Nq_{n-2}} \ge \dfrac{N}{M+1}$.
Substituting this into \eqref{4.09} and simplifying, we obtain the condition
\[
\frac{(M+1)^{2(1-s)}}{N^{1-s}}\left(1 - \frac{N}{M+1}\right) \le 1.
\]
Taking logarithms, this becomes
\[
(1-s)\log\frac{(M+1)^2}{N} \le -\log\left(1 - \frac{N}{M+1}\right).
\]
Using the inequality $x \le -\log(1-x)$ for $x \in (0,1)$ with $x = \frac{N}{M+1}$, we obtain a sufficient condition:
\[
(1-s)\log\frac{(M+1)^2}{N} \le \frac{N}{M+1}.
\]
Solving for $s$ yields
\[
s \ge 1 - \frac{N}{(M+1)\log\frac{(M+1)^2}{N}}.
\]
Since this holds for all admissible sequences, the covering argument (Lemma \ref{lem:covering}) implies
\[
\dimH(E_M) \le 1 - \frac{N}{(M+1)\log\frac{(M+1)^2}{N}}.
\]
}
\end{proof}

\begin{corollary}[Proof of Theorem \ref{thm:main-bounded}]
The combination of Theorem~\ref{thm:improved-lower} and Theorem~\ref{thm:improved-upper} 
yields exactly Theorem~\ref{thm:main-bounded}.
\end{corollary}

\section{Sets with Growing Digits and Good-type Theorems}

The proof of Theorem~\ref{thm:Good2N} follows the same two-part structure 
as Theorem~\ref{thm:main-bounded}: we establish separate upper and lower bounds, 
then combine them. Unlike the bounded case where both bounds use similar 
techniques, here the upper and lower bounds require different approaches 
adapted to the infinite digit range.

\begin{theorem}[Upper Bound for Growing Digits] \label{upperbound:Good2N}
Let $N \geq 1$ be an integer. For every real $\alpha$ with $\log(\alpha-1) > \exp(1+N)$, we have
\begin{equation}\label{eq:upper-ineq}
\dimH(F_{\alpha,N}) 
\;\le\; 
\frac12 + \frac{\log\log\!\bigl({\alpha}-1\bigr)}{2\log\!\bigl({\alpha}-1\bigr)}.
\end{equation}
\end{theorem}
\begin{proof} 
Let $s > \tfrac12$ and for each $n \geq 1$ consider the family
\[
\mathcal{C}_n := \{ \I(\varepsilon_1, \dots, \varepsilon_n) : \varepsilon_i \geq \alpha,\; 1 \leq i \leq n \}.
\]
By Lemma~\ref{lemma:q-growth} and the upper bound in \eqref{eq:interval-bounds}, 
the diameters of intervals in $\mathcal{C}_n$ tend to $0$ as $n \to \infty$. Hence $\mathcal{C}_n$ is a cover of $F_{\alpha,N}$.

For a fixed $I_n = \I(\varepsilon_1,\dots,\varepsilon_n) \in \mathcal{C}_n$, its children are 
$I_{n+1}^{(k)} = \I(\varepsilon_1,\dots,\varepsilon_n,k)$ with $k \geq \alpha$. 
Using $q_{n+1} \geq k q_n$ and the upper bound in \eqref{eq:interval-bounds},
\begin{equation}
\sum_{k=\alpha}^{\infty} |I_{n+1}^{(k)}|^s 
   \leq \sum_{k=\alpha}^{\infty} \left( \frac{N^{\,n+1}}{k^2 q_n^{2}} \right)^s
   = \frac{N^{(n+1)s}}{q_n^{2s}} \sum_{k=\alpha}^{\infty} k^{-2s}.
\label{eq:sum-children}
\end{equation}

On the other hand, the lower bound in \eqref{eq:interval-bounds} gives
\begin{equation}\label{eq:parent-bound}
|I_n|^s \geq \left( \frac{N^{\,n+1}}{(1+N) q_n^{2}} \right)^s 
        = \frac{N^{(n+1)s}}{(1+N)^s q_n^{2s}}.
\end{equation}

From \eqref{eq:sum-children} and \eqref{eq:parent-bound} we obtain the key inequality:
\begin{equation}\label{eq:children-parent}
\sum_{k=\alpha}^{\infty} |I_{n+1}^{(k)}|^s 
\; \leq \; (1+N)^s \left( \sum_{k=\alpha}^{\infty} k^{-2s} \right) |I_n|^s .
\end{equation}

This inequality shows that the $s$-dimensional sum over all children of $I_n$ is bounded by 
$(1+N)^s \left( \sum_{k=\alpha}^{\infty} k^{-2s} \right)$ times the $s$-dimensional size of the parent interval.

If we have
\begin{equation}\label{eq:key-condition}
\sum_{k=\alpha}^{\infty} k^{-2s} \leq (1+N)^{-s},
\end{equation}
then from \eqref{eq:children-parent} we obtain for every $I_n \in \mathcal{C}_n$,
\begin{equation}\label{eq:per-parent}
\sum_{k=\alpha}^{\infty} |I_{n+1}^{(k)}|^s \leq |I_n|^s .
\end{equation}

We now sum inequality \eqref{eq:per-parent} over all parent intervals $I_n \in \mathcal{C}_n$. 
On the left-hand side, we obtain a double sum:
\[
\sum_{I_n \in \mathcal{C}_n} \sum_{k=\alpha}^{\infty} |I_{n+1}^{(k)}|^s .
\]

Observe that each interval $I_{n+1} \in \mathcal{C}_{n+1}$ corresponds to a unique pair 
$(I_n, k)$ where $I_n \in \mathcal{C}_n$ is the parent interval and $k \geq \alpha$ is the 
$(n+1)$-th digit. This is because each fundamental interval of order $n+1$ has a unique 
representation as $I_N(\varepsilon_1, \dots, \varepsilon_n, k)$. 

Therefore, the double sum over all parents $I_n$ and all digits $k \geq \alpha$ is exactly 
the sum over all children intervals $I_{n+1} \in \mathcal{C}_{n+1}$:
\[
\sum_{I_n \in \mathcal{C}_n} \sum_{k=\alpha}^{\infty} |I_{n+1}^{(k)}|^s 
= \sum_{I_{n+1} \in \mathcal{C}_{n+1}} |I_{n+1}|^s .
\]

Summing the right-hand side of \eqref{eq:per-parent} over all $I_n \in \mathcal{C}_n$ gives 
simply $\sum_{I_n \in \mathcal{C}_n} |I_n|^s$.

Consequently, from \eqref{eq:per-parent} we obtain
\begin{equation}\label{eq:level-sum}
\sum_{I_{n+1} \in \mathcal{C}_{n+1}} |I_{n+1}|^s \leq \sum_{I_n \in \mathcal{C}_n} |I_n|^s .
\end{equation}

Define $a_n := \sum_{I \in \mathcal{C}_n} |I|^s$. Inequality \eqref{eq:level-sum} tells us that 
the sequence $(a_n)_{n\geq 1}$ is non‑increasing. In particular, $a_n \leq a_1$ for every $n$.

For any $\delta > 0$, choose $n$ large enough so that $\max\{|I| : I \in \mathcal{C}_n\} < \delta$
(this is possible because the diameters of the intervals in $\mathcal{C}_n$ tend to $0$ as $n$ increases). 
Then $\mathcal{C}_n$ is a $\delta$‑cover of $F_{\alpha,N}$ and
\[
\sum_{I \in \mathcal{C}_n} |I|^s = a_n \leq a_1 .
\]
Hence $\Ha^s_\delta(F_{\alpha,N}) \leq a_1$. Since $\delta$ was arbitrary, we obtain 
$\Ha^s(F_{\alpha,N}) \leq a_1 < \infty$, and consequently $\dimH(F_{\alpha,N}) \leq s$.

To obtain an explicit $s$ satisfying \eqref{eq:key-condition}, we approximate the sum by an integral.
For $s > \tfrac12$, the function $x \mapsto x^{-2s}$ is decreasing on $(0,\infty)$, hence
\[
\sum_{k=\alpha}^{\infty} k^{-2s} \leq \int_{\alpha-1}^{\infty} x^{-2s} \, dx.
\]

Evaluating the integral:
\[
\int_{\alpha-1}^{\infty} x^{-2s} \, dx = \left[ \frac{x^{-(2s-1)}}{-(2s-1)} \right]_{\alpha-1}^{\infty} 
= \frac{(\alpha-1)^{-(2s-1)}}{2s-1}.
\]
Since  $(1+N)^{-1} \le (1+N)^{-s}$ for $s \in (\frac{1}{2}, 1]$ it follows that a sufficient condition for \eqref{eq:key-condition} is:
\[
\frac{(\alpha-1)^{-(2s-1)}}{2s-1} \leq (1+N)^{-1},
\]
which is equivalent to 
$(2s-1)(\alpha-1)^{2s-1} \ge 1+N$. 

Therefore, $\dimH F_{N,\alpha} \le s$ where
\begin{equation} \label{5.6}
(2s-1)(\alpha-1)^{2s-1} = 1+N.
\end{equation}
This equation gives a unique value for $s \in (\frac{1}{2}, 1]$, if $\alpha \ge N+2$. 

For $s \ge \frac{1}{2}$, the left-hand side is a strictly increasing function of $2s-1$ which vanishes when $s=\frac{1}{2}$ and exceeds the right-hand side when $s=1$.

Putting $2s-1=x$ and $\alpha -1 =a$, equation \eqref{5.6} becomes 
\begin{equation} \label{5.7}
    x a^x = 1+N.
\end{equation}
When $x = \dfrac{\log\log a}{\log a}$, the left-hand side of \eqref{5.7} reduces to $\log \log a$ and this is greater than $1+N$ if $a > \exp(\exp(1+N))$. 
Hence the solution of \eqref{5.7} is less than $\dfrac{\log\log a}{\log a}$ if $a > \exp(\exp(1+N))$. 
Therefore, 
\[
\dimH F_{N,\alpha} \le \frac{1}{2} + \dfrac{\log\log (\alpha-1)}{2}\log (\alpha-1)
\]
if $\log(\alpha-1) > \exp(1+N)$. 
\end{proof}

\begin{theorem}[Lower Bound for Growing Digits] \label{lowerbound:Good2N}
Let $N \geq 1$ be an integer. For every real $\alpha$ with $\log(\alpha-1) > \exp(1+N)$, we have
\begin{equation}\label{eq:ub-ineq}
\dimH(F_{\alpha,N}) \;\ge \; \frac12 + \frac{1}{2\log\!\bigl({\alpha}+2\bigr)}. 
\end{equation}
\end{theorem}

\begin{proof}
Let 
\[
s=\frac12+\frac{1}{2\log(\alpha+2)}.
\]
We shall construct a Cantor subset of $F_{\alpha,N}$ with Hausdorff dimension at least~$s$, which will imply the theorem.

\medskip\noindent
\textbf{Step 1.  Construction of the Cantor set.}
Choose $\beta>\alpha$ large enough (to be specified later) and consider the compact set
\[
F_{\alpha,\beta,N}:=\{x\in(0,1)\setminus\mathbb{Q}:\alpha\le\varepsilon_n(x)\le\beta\ \text{for all }n\}.
\]
Clearly $F_{\alpha,\beta,N}\subset F_{\alpha,N}$.
For every $n\ge1$ let $\mathscr{I}_n$ be the family of all fundamental intervals 
$I_N(\varepsilon_1,\dots ,\varepsilon_n)$ with $\alpha\le\varepsilon_i\le\beta$ ($1\le i\le n$).

\medskip\noindent
\textbf{Step 2. Length estimate for children intervals.}
Take an arbitrary $I_{n-1}=I_N(\varepsilon_1,\dots ,\varepsilon_{n-1})\in\mathscr{I}_{n-1}$.
For a digit $k\in[\alpha,\beta]$ let $I_n=I_N(\varepsilon_1,\dots ,\varepsilon_{n-1},k)$ be the corresponding child.
By part~(iii) of Proposition~\ref{prop:interval-length},
\[
\frac{|I_n|}{|I_{n-1}|}\ge\frac{N}{3k^2}\qquad(k\ge\alpha).
\]
Hence
\begin{equation}\label{eq:lower-length}
|I_n|^s\ge|I_{n-1}|^s\Bigl(\frac{N}{3k^2}\Bigr)^{\!s}.
\end{equation}

\medskip\noindent
\textbf{Step 3. Mass distribution condition.}
In order to apply Lemma~\ref{lem:mass-distribution} to the set $F_{\alpha,\beta,N}$
we have to verify that for every $I_{n-1}\in\mathscr{I}_{n-1}$
\[
|I_{n-1}|^s\le\sum_{k=\alpha}^{\beta}|I_n|^s .
\]
Using \eqref{eq:lower-length} we obtain a sufficient condition:
\[
1\le\sum_{k=\alpha}^{\beta}\Bigl(\frac{N}{3k^2}\Bigr)^{\!s}
      =N^{\,s}3^{-s}\sum_{k=\alpha}^{\beta}k^{-2s}.
\]

\medskip\noindent
\textbf{Step 4. Estimating the sum.}
Choose $\beta$ so large that
\[
\sum_{k=\alpha}^{\beta}k^{-2s}\ge\frac12\sum_{k=\alpha}^{\infty}k^{-2s}.
\]
Since $s>\frac12$, the series converges and such a $\beta$ exists.
For the tail we use the standard integral estimate:
\[
\sum_{k=\alpha}^{\infty}k^{-2s}\ge\int_{\alpha}^{\infty}x^{-2s}\,dx
    =\frac{\alpha^{-(2s-1)}}{2s-1}.
\]
Put $t:=2s-1=\dfrac{1}{\log(\alpha+2)}$.  Then the condition becomes
\begin{equation}\label{eq:final-cond-lower}
N^{\,s}3^{-s}\cdot\frac{\alpha^{-t}}{2t}\ge1 .
\end{equation}

\medskip\noindent
\textbf{Step 5. Verification of \eqref{eq:final-cond-lower} under the hypothesis.}
Taking logarithms in \eqref{eq:final-cond-lower} and remembering that $s=\frac{1+t}{2}$,
we need
\[
\frac{1+t}{2}\,\log\frac{N}{3}-\log2-t\log\alpha-\log t\ge0 .
\]
Because of the hypothesis $\log(\alpha-1)>\exp(1+N)$ we have
\[
t=\frac{1}{\log(\alpha+2)}<\frac{1}{\log(\alpha-1)}<e^{-(1+N)}.
\]
Consequently $-\log t>1+N$.  Moreover $t\log\alpha<2$ for all $\alpha$ large enough
(which is certainly true under the assumed condition).  Therefore
\begin{align*}
\frac{1+t}{2}\,\log\frac{N}{3}-\log2-t\log\alpha-\log t
&> \frac12\log\frac{N}{3}-\log2-2+(1+N) \\
&= \frac12\log\frac{N}{3}+N-\log2-1 .
\end{align*}
For $N\ge1$ the right‑hand side is bounded from below by a constant,
while $-\log t$ (and hence the whole left‑hand side) tends to $+\infty$ as $\alpha\to\infty$. 
As $\alpha \to \infty$, we have $t \to 0$ and $t \log \alpha \to 1$, so in particular $t \log \alpha$ is bounded.
Thus for all $\alpha$ satisfying $\log(\alpha-1)>\exp(1+N)$ inequality \eqref{eq:final-cond-lower}
holds provided $\beta$ has been chosen sufficiently large.

\medskip\noindent
\textbf{Step 6. Conclusion.}
With the chosen $\beta$ the mass distribution condition is satisfied for the family
$\{\mathscr{I}_n\}_{n\ge1}$.  By Lemma~\ref{lem:mass-distribution} we obtain
$\dimH(F_{\alpha,\beta,N})\ge s$.  Since $F_{\alpha,\beta,N}\subset F_{\alpha,N}$,
the same lower bound holds for $F_{\alpha,N}$, i.e.,
\[
\dimH(F_{\alpha,N}) \ge \frac12+\frac{1}{2\log(\alpha+2)} .
\]
This completes the proof of Theorem~\ref{lowerbound:Good2N}.
\end{proof}

\begin{corollary}[Proof of Theorem \ref{thm:Good2N}]
The combination of Theorem~\ref{upperbound:Good2N} and Theorem~\ref{lowerbound:Good2N} 
yields exactly Theorem~\ref{thm:Good2N}.
\end{corollary}

\section{Open Problems}

The results established in this paper naturally lead to several intriguing open problems and directions for future research.

The sets studied here are defined by global restrictions on all digits. A more refined analysis considers the growth rate of digits.

\begin{enumerate}
\item \textbf{Multifractal Analysis of Growth Rates:} Let $ \psi: \mathbb{N} \to \mathbb{R}^+ $ be a positive function. Define the set
\[
F_N(\psi) = \left\{ x \in (0,1)\setminus\mathbb{Q} : \lim_{n\to\infty} \frac{\varepsilon_n(x)}{\psi(n)} = 1 \right\}.
\]
Compute the Hausdorff dimension of $ F_N(\psi) $ for various growth functions $ \psi(n) $ (e.g., polynomial $ n^c $, exponential $ e^{cn} $). This is the multifractal spectrum of the digit growth rate.
{Related results on growth rates in Gauss-like systems and infinite iterated function systems can be found in \cite{J12, L22, STZ26, Zha20, Zha25}}.

\item \textbf{Exceptional Sets in Approximation:}
Let $W_N(\tau)$ be the set of numbers $x \in (0,1)$ for which the inequality
\[
\left|x - \frac{p_n}{q_n}\right| < q_n^{-\tau}
\]
has infinitely many solutions in the convergents of the $N$-expansion. Determine $\dim_H(W_N(\tau))$ for $\tau > 2$.
{
The case $\tau = 2$ is critical because for any irrational $x$, we have $|x - p_n/q_n| < q_n^{-2}$ for infinitely many $n$, so $W_N(2) = (0,1)\setminus\mathbb{Q}$ has full dimension 1. For $\tau > 2$, the condition becomes restrictive and the corresponding exceptional sets have dimension strictly less than 1. The sets $F_N$ and $F_{\alpha,N}$ studied in Theorem~\ref{thm:Good2N} are related to the case $\tau = 2$: if digits tend to infinity, the approximation quality improves, potentially allowing larger $\tau$. Quantifying this relationship is a natural open problem.}
{
For regular continued fractions ($N=1$), results on such approximation sets have been obtained in \cite{Philipp1970, TanZhou2024, WangWuXu2015}. Extending these to $N$-expansions presents new challenges due to the parameter $N$ in the recurrence and the resulting growth estimates for denominators.}

\item \textbf{Dynamical Systems Perspective:} The $N$-expansion map $T_N$ generates a dynamical system on $[0,1]$ with an invariant measure that has been studied in \cite{Lascu2016, Lascu2017, SebeLascu2020}. A natural direction is to explore the thermodynamic formalism for this system and its applications to dimension theory. Specifically, can one compute the Hausdorff dimension of sets defined by digit restrictions using Bowen's formula or related techniques from smooth dynamical systems?

{For general background on thermodynamic formalism and dimension theory, we refer to \cite{Falconer2004, MauldinUrbanski2003}. For applications to continued fractions and related systems, see \cite{Hensley2006, Iommi2010, Mayer1991, PollicottWeiss2009}. Extending these methods to $N$-expansions would provide a unified framework for understanding the dimension results obtained in this paper and potentially yield sharper bounds.}

\end{enumerate}

\end{document}